\documentclass[reqno,11pt]{amsart}
\usepackage{amsfonts}
\usepackage{amssymb}
\usepackage{graphicx}
\usepackage{pstricks}
\usepackage{amsmath}
\usepackage{amsmath}
\usepackage{amsxtra}

\theoremstyle{plain}
\newtheorem{theorem}{Theorem}[section]
\newtheorem{lemma}[theorem]{Lemma}

\theoremstyle{definition}

\newtheorem{remark}[theorem]{Remark}

\numberwithin{equation}{section}
\begin{document}

\title[An answer to a question of A. Lubin]{An answer to a question of A.
Lubin: The lifting problem for commuting subnormals}
\author{Sang Hoon Lee}
\address{Department of Mathematics, Chungnam National University, Daejeon,
305-764, Korea}
\email{slee@cnu.ac.kr}

\urladdr{}
\author{Woo Young Lee}
\address{Department of Mathematics, Seoul National University, Seoul,
151-747, Korea}
\email{wylee@snu.ac.kr}

\urladdr{}
\author{Jasang Yoon}
\address{School of Mathematical and Statistical Sciences,
The University of Rio Grande Valley, Edinburg, Texas 78539}
\email{jasang.yoon@utrgv.edu}

\thanks{
The first named author was partially supported by Basic Science Research
Program through the National Research Foundation of Korea(NRF) funded by the
Ministry Education, Science and Technology (2013R1A1A2008640). \\
The second named author was partially supported by Basic Science Research
Program through the National Research Foundation of Korea(NRF) funded by the
Ministry Education, Science and Technology (2009-0085279). \\
The third named author was partially supported by a Faculty Research Council
Grant at The University of Texas-Pan American.}


\begin{abstract}
In this paper we give an answer to a long-standing open question on the lifting
problem for commuting subnormals (due to A. Lubin): The subnormality for the
sum of commuting
subnormal operators does not guarantee the existence of commuting normal
extensions. \vskip .7cm
\end{abstract}

\maketitle


\section{\label{Introduction} Introduction}

\S 1. \textbf{A historical background}. \ The \textit{Lifting Problem for
Commuting Subnormals} (LPCS) asks for necessary and sufficient conditions
for a pair of commuting subnormal operators on a Hilbert space to admit
commuting normal extensions. \ This is an old problem in operator theory. \
The aim of this paper is to answer a long-standing open problem about the
LPCS. \

To begin with, let $\mathcal{H}$ denote a complex Hilbert space and $%
\mathcal{B(H)}$ denote the set of all bounded linear operators acting on $%
\mathcal{H}$. \ For an operator $T\in\mathcal{B(H)}$, $T^*$ denotes the
adjoint of $T$. \ An operator $T\in \mathcal{B(H)}$ is said to be \textit{%
normal} if $T^*T=TT^*$, \textit{hyponormal} if its self-commutator $%
[T^*,T]\equiv T^*T-TT^*$ is positive semi-definite, and \textit{subnormal}
if there exists a Hilbert space $\mathcal{K}$ containing $\mathcal{H}$ and a
normal operator $N$ on $\mathcal{K}$ such that $N \mathcal{H} \subseteq
\mathcal{H}$ and $T=N\vert_{\mathcal{H}}$, a restriction of $N$ to $\mathcal{%
H}$. \ In this case, $N$ is called a \textit{normal extension} of $T$. \ In
1950, P.R. Halmos \cite{Hal2} introduced the notion of a subnormal operator
for the purpose of the study of dilations and extensions of operators on a
Hilbert space. \ Nowadays, the theory of subnormal operators has become an
extensive and highly developed area, which has made significant
contributions to a number of problems in functional analysis, operator
theory, mathematical physics, and several other fields.

We recall that if $\mathfrak{A}$ is a subset of $\mathcal{B(H)}$ then
the \textit{commutant} of $\mathfrak{A}$, denoted $\mathfrak{A}^{\prime }$,
is the set of operators in $\mathcal{B(H)}$ which commute with every
operator in $\mathfrak{A}$. \ If $T\in \mathcal{B(H)}$ is a subnormal
operator and $N$ is a normal extension of $T$, then we say that an operator $%
A$ in $\{T\}^{\prime }$ \textit{lifts} to $\{N\}^{\prime }$ if there exists
an operator $B$ in $\{N\}^{\prime }$ such that $B(\mathcal{H})\subseteq
\mathcal{H}$ and $A=B|_{\mathcal{H}}$. \ In 1971, J.A. Deddens \cite{De}
provided an example that not every operator in $\{T\}^{\prime }$ lifts to $%
\{N\}^{\prime }$. \ As an interesting inquiry in the commutant lifting
problem, an old problem (LPCS) in operator theory has been brought up: for
two commuting subnormal operators $T_{1}$ and $T_{2}$, find necessary and
sufficient conditions for a pair of $T_{1}$ and $T_{2}$ to admit commuting
normal extensions. \ The LPCS has been studied by many authors including
\cite{Ab}, \cite{Ab2}, \cite{AbD}, \cite{Br}, \cite{bridge}, \cite{CLY1}, \cite{CLY7}, \cite%
{CuYo1}, \cite{Fra}, \cite{Ito}, \cite{Lu3}, \cite{Lu1}, \cite{Lu2}, \cite%
{Mla}, \cite{Olin}, \cite{Slo}, \cite{Szy}, \cite{Yos}, etc. \ There are
many known examples of commuting pairs of subnormal operators which admit no
lifting (cf. M.B. Abrahamse \cite{Ab} and A.R. Lubin \cite{Lu3}). \ Also,
many sufficient conditions for the existence of a lifting have been found. \
For instance, a commuting pair of subnormal operators $T_{1}$ and $T_{2}$
admits a lifting if either $T_{1}$ or $T_{2}$ is normal (J. Bram \cite{Br}),
if either $T_{1}$ or $T_{2}$ is cyclic (T. Yoshino \cite{Yos}), if either $%
T_{1}$ or $T_{2}$ is an isometry (M. Slocinski \cite{Slo}), or if
the spectrum of either $T_{1}$ or $T_{2}$ is finitely connected and
the spectrum of its minimal normal extension is contained in the
boundary of its spectrum. \ On the other hand, in all of the known
examples of the absence of lifting, the key property missing is the
subnormality of $T_{1}+T_{2}$. \ Indeed, in 1978, A.R. Lubin
\cite{Lu1} addressed a concrete problem about the LPCS: if $T_1$ and
$T_2$ are commuting subnormal operators, do they admit commuting
normal extensions when $p(T_{1},T_{2})$ is subnormal for every
2-variable polynomial $p$, or more weakly, when $T_{1}+T_{2}$ is
subnormal\thinspace ? \ In 1994, E. Franks \cite{Fra} showed that
the first condition gives an affirmative answer; indeed, commuting
subnormal operators $T_{1}$ and $T_{2}$ admit commuting normal
extensions if $p(T_{1},T_{2})$ is subnormal for each 2-variable
polynomial $p$ of degree at most $5$. \ However, the second
condition still remains open: that is, if $T_1$ and $T_2$ are
commuting subnormal operators,
\begin{equation}\label{P-Lubin}
\hbox{does the subnormality
of $T_1+T_2$ guarantee commuting normal extensions of $T_1$ and
$T_2$\,?}
\end{equation}
What is the reason why 36 years passed while question
(\ref{P-Lubin}) remained unanswered ? \ The difficulty of
determining the
subnormality of $T_1+T_2$ is one explanation for failing to
answer question (\ref{P-Lubin}). \ Probably, the most effective way
to determine the subnormality of $T_1+T_2$ is Agler's criterion for
subnormality in \cite{Agl}.  \ However, in view of Lambert's Theorem \cite{Lam},
a main ingredient to examine the
subnormality is weighted shifts and Agler's criterion for the
weighted shifts involves quite intricately combinatorial
expressions, which are hard problems to solve. \
Thus, we had to develop the theory of 2-variable weighted shifts before
the time is ripe for answering question (\ref{P-Lubin}). \ In this
paper, we give a negative answer to question (\ref{P-Lubin}), by
using 2-variable weighted shifts together with the
disintegration-of-measure technique and ingenious combinatorial
computations.


\medskip

\noindent \S 2. \textbf{Joint subnormality}. \ The notion of joint
hyponormality for the general case of $n$-tuples of operators was first
formally introduced by A. Athavale \cite{Ath}. \
Joint hyponormality originated from the LPCS, and it has also been
considered with an aim at understanding the gap between hyponormality and
subnormality for single operators. \ In some sense, the birth of joint
hyponormality occurred with the Bram-Halmos theorem for subnormality of an
operator. \ The Bram-Halmos criterion for subnormality (cf. \cite{Br}, \cite%
{Con}) states that an operator $T\in \mathcal{B(H)}$ is subnormal if and
only if $\sum_{i,j}(T^{i}x_{j},T^{j}x_{i})\geq 0$ for all finite collections
$x_{0},x_{1},\cdots ,x_{k}\in \mathcal{H}$. \
Given an $n$-tuple $\mathbf{T}\equiv (T_{1},%
\hdots,T_{n})$ of operators on $\mathcal{H}$, we let $[\mathbf{T}^{\ast },%
\mathbf{T}]\label{bfT^*bfT}\in \mathcal{B(H\oplus \cdots \oplus H)}$ denote
the \textit{self-commutator} of $\mathbf{T}$, defined by
\begin{equation*}
\lbrack \mathbf{T}^{\ast },\mathbf{T}]:=%
\begin{pmatrix}
\hbox{$[T_1^*, T_1]$} & \hbox{$[T_2 ^*, T_1]$} & \hdots & %
\hbox{$[T_n^*,T_1]$} \\
\hbox{$[T_1^*, T_2]$} & \hbox{$[T_2 ^*, T_2]$} & \hdots & %
\hbox{$[T_n^*,T_2]$} \\
\vdots & \vdots & \ddots & \vdots \\
\hbox{$[T_1^*,T_n]$} & \hbox{$[T_2^*,T_n]$} & \hdots & \hbox{$[T_n^*,T_n]$}%
\end{pmatrix}%
,
\end{equation*}%
where $[S,T]:=ST-TS$ for $S,T\in \mathcal{B}(\mathcal{H})$. \ \ By analogy
with the case $n=1$, we shall say (\cite{Ath}, \cite{CMX}) that $\mathbf{T}$
is \textit{jointly hyponormal} (or simply, \textit{hyponormal}) if $[\mathbf{%
T}^{\ast },\mathbf{T}]$ is a positive operator on $\mathcal{H}\oplus \cdots
\oplus \mathcal{H}$. \
Thus, the
Bram-Halmos criterion can be restated as: $T\in \mathcal{B(H)}$ is subnormal
if and only if $(T,T^{2},\cdots ,T^{k})$ is hyponormal for every $k\in
\mathbb{Z}_{+}$. \
The $n$-tuple $\mathbf{T}\equiv (T_{1},%
\hdots,T_{n})$ is said to be (\textit{%
jointly}) \textit{normal} if $\mathbf{T}$ is commuting and every $T_{i}$ is
a normal operator and is said to be (\textit{jointly}) \textit{subnormal} if $\mathbf{T}$ is the restriction
of a normal $n$-tuple to a common invariant subspace, i.e., $\mathbf T$ admits commuting normal extensions. \
Thus the LPCS can be restated as:

\smallskip

\noindent
{\bf LPCS}: \textit{Find necessary
and sufficient conditions for a commuting pair of subnormal operators to be
subnormal.}


\bigskip

\noindent \S 3. \textbf{A main ingredient of the paper - two variable weighted
shifts}. \ To answer question (\ref{P-Lubin}), we exploit $2$-variable
weighted shifts as a main tool. \ It is well known that the subnormality of
an arbitrary operator can be ascertained by examining the subnormality of an
associated family of weighted shifts \cite{Lam}. \ Thus, single and
multivariable weighted shifts have played an important role in the study of
the LPCS. \ They have also played a significant role in the study of
cyclicity and reflexivity, in the study of $C^{\ast }$-algebras generated by
multiplication operators on Bergman spaces, as fertile ground to test new
hypotheses, and as canonical models for theories of dilation and positivity.
\ We review the definition and basic properties of $2$-variable
weighted shifts.

Recall that given a bounded sequence of positive numbers $\alpha :\alpha
_{0},\alpha _{1},\cdots $ (called weights or a weight sequence), the
(unilateral) weighted shift $W_{\alpha }$ associated with the sequence
$\alpha $ is the operator
on $\ell ^{2}(\mathbb{Z}_{+})$ defined by $W_{\alpha }e_{n}:=\alpha
_{n}e_{n+1}$ for all $n\geq 0$, where $\{e_{n}\}_{n=0}^{\infty }$ is the
canonical orthonormal basis for $\ell ^{2}(\mathbb{Z}_{+})$. \
 We shall often write $\mathrm{shift}(\alpha
_{0},\alpha _{1},\cdots )$ to denote the weighted shift $W_{\alpha }$ with a
weight sequence $\alpha \equiv \{\alpha _{n}\}_{n=0}^{\infty }$. \
The {\it moments} of $\alpha$ are defined by
$$
\gamma_k\equiv \gamma_k(\alpha):=\alpha_0^2\cdots\alpha_{k-1}^2\quad (k\ge 1)
$$
and $\gamma_0:=1$. \
There is a well-known criterion of subnormality of weighted shifts,
due to C. Berger (cf. \cite[III.8.16]{Con})  and
independently established by R. Gellar and L.J. Wallen \cite{GeWa}:
$W_\alpha$ is subnormal if and only if there exists a probability measure
$\xi_\alpha$ supported in $[0,||W_\alpha||^2]$
(called the {\it Berger measure} of $W_\alpha$)
such that $\gamma_k(\alpha)=\int s^k d\xi_\alpha(s)$ ($k\ge 1$). \
If $W_{\alpha }$ is subnormal with Berger measure $%
\xi _{\alpha }$ and $i\geq 1$, and if we let $\mathcal{L}_{i}:=\bigvee
\{e_{n}:n\geq i\}$ denote the invariant subspace obtained by removing the
first $i$ vectors in the canonical orthonormal basis of $\ell ^{2}(\mathbb{Z}%
_{+})$, then
\begin{equation}
\hbox{the Berger
measure of $W_{\alpha }|_{\mathcal{L}_{i}}$ is $\frac{s^{i}}{\gamma
_{i}(\alpha)} d\xi_{\alpha} (s)$},  \label{Berger_restrict}
\end{equation}%
where $W_{\alpha }|_{\mathcal{L}_{i}}$ denotes the restriction of $W_{\alpha
}$ to $\mathcal{L}_{i}$. \

We now consider two bounded double-indexed sequences $\alpha \equiv \{\alpha _{%
\mathbf{k}}\},\beta \equiv \{\beta _{\mathbf{k}}\}\in \ell ^{\infty }(%
\mathbb{Z}_{+}^{2})$, $\mathbf{k}\equiv (k_{1},k_{2})\in \mathbb{Z}_{+}^{2}:=%
\mathbb{Z}_{+}\times \mathbb{Z}_{+}$ and let $\ell ^{2}(\mathbb{Z}_{+}^{2})$%
\ be the Hilbert space of square-summable complex sequences indexed by $%
\mathbb{Z}_{+}^{2}$. \ (Note that $\ell ^{2}(\mathbb{Z}_{+}^{2})$ is
canonically isometrically isomorphic to $\ell ^{2}(\mathbb{Z}_{+})\bigotimes
\ell ^{2}(\mathbb{Z}_{+})$.) \ We define a $2$-variable weighted shift $%
W_{(\alpha, \beta)} \equiv (T_{1},T_{2})$, a pair of $T_{1}$ and $T_{2}$ on $\ell ^{2}(%
\mathbb{Z}_{+}^{2})$, by 
$T_{1}e_{\mathbf{k}}:=\alpha _{\mathbf{k}}e_{\mathbf{k+}\varepsilon _{1}}%
\text{ and }T_{2}e_{\mathbf{k}}:=\beta _{\mathbf{k}}e_{\mathbf{k+}%
\varepsilon _{2}}$, 
where $\mathbf{\varepsilon }_{1}:=(1,0)$, $\mathbf{\varepsilon }_{2}:=(0,1)$,
and $\{e_{\mathbf{k}}\}_{\mathbf{k}\in \mathbb{Z}_{+}^{2}}$ denotes the
canonical orthonormal basis of $\ell ^{2}(\mathbb{Z}_{+}^{2})$ (see Figure
1(i)). \ Clearly,
\begin{equation}
T_{1}T_{2}=T_{2}T_{1}\Longleftrightarrow \beta _{\mathbf{k+}\varepsilon
_{1}}\alpha _{\mathbf{k}}=\alpha _{\mathbf{k+}\varepsilon _{2}}\beta _{%
\mathbf{k}}\;(\text{all }\mathbf{k}\in \mathbb{Z}_{+}^{2}).
\label{commuting}
\end{equation}%
In the sequel, we assume that all $2$-variable weighted shifts $W_{(\alpha,\beta )}$
are commuting, i.e., it satisfies the condition (\ref{commuting}%
). \ Given $\mathbf{k}\equiv (k_{1},k_{2})\in \mathbb{Z}_{+}^{2}$, the
moment of order $\mathbf{k}$ for a pair $(\alpha ,\beta )$ satisfying (\ref%
{commuting}) is defined by
\begin{equation*}
\gamma _{\mathbf{k}}\equiv \gamma _{\mathbf{k}}(\alpha,\beta):=%
\begin{cases}
1 & \text{if }k_{1}=0\text{ and }k_{2}=0; \\
\alpha _{(0,0)}^{2}\cdots \alpha _{(k_{1}-1,0)}^{2} & \text{if }k_{1}\geq 1%
\text{ and }k_{2}=0; \\
\beta _{(0,0)}^{2}\cdots \beta _{(0,k_{2}-1)}^{2} & \text{if }k_{1}=0\text{
and }k_{2}\geq 1; \\
\alpha _{(0,0)}^{2}\cdots \alpha _{(k_{1}-1,0)}^{2}\beta
_{(k_{1},0)}^{2}\cdots \beta _{(k_{1},k_{2}-1)}^{2} & \text{if }k_{1}\geq 1%
\text{ and }k_{2}\geq 1.%
\end{cases}%
\end{equation*}%
We note that, due to the commutativity condition (\ref{commuting}), $%
\gamma _{\mathbf{k}}$ can be computed using any nondecreasing path from $%
(0,0)$ to $\mathbf{k}$. \
We recall that there is a 2-variable Berger's Theorem, due to
N. Jewell and A.R. Lubin \cite{JeLu}: a 2-variable weighted shift
$W_{(\alpha,\beta)}\equiv (T_{1},T_{2})$ is subnormal
 if and only if there exists a probability measure $\mu $ (called
\textit{Berger measure} of $W_{(\alpha ,\beta )}$) defined on the $2$%
-dimensional rectangle $R=[0,||T_{1}||^{2}]\times \lbrack 0,||T_{2}||^{2}]$
such that
\begin{equation*}
\gamma _{\mathbf{k}}(\alpha, \beta)=\iint_{R}s^{k_{1}}t^{k_{2}}d\mu
(s,t)\quad \hbox{for all $\mathbf{k}\equiv(k_1, k_2)\in \mathbb{
Z}_{+}^{2}$ (called Berger's Theorem)}.
\end{equation*}%

\vskip 1cm


\setlength{\unitlength}{1mm} \psset{unit=1mm}
\begin{figure}[th]
\begin{center}
\begin{picture}(135,68)

\psline{->}(20,20)(70,20)
\psline(20,35)(68,35)
\psline(20,50)(68,50)
\psline(20,65)(68,65)
\psline{->}(20,20)(20,70)
\psline(35,20)(35,68)
\psline(50,20)(50,68)
\psline(65,20)(65,68)

\put(12,16){\footnotesize{$(0,0)$}}
\put(31.5,16){\footnotesize{$(1,0)$}}
\put(46.5,16){\footnotesize{$(2,0)$}}
\put(61.5,16){\footnotesize{$(3,0)$}}

\put(25,21){\footnotesize{$\alpha_{(0,0)}$}}
\put(40,21){\footnotesize{$\alpha_{(1,0)}$}}
\put(55,21){\footnotesize{$\alpha_{(2,0)}$}}
\put(66,21){\footnotesize{$\cdots$}}

\put(25,36){\footnotesize{$\alpha_{(0,1)}$}}
\put(40,36){\footnotesize{$\alpha_{(1,1)}$}}
\put(55,36){\footnotesize{$\alpha_{(2,1)}$}}
\put(66,36){\footnotesize{$\cdots$}}

\put(25,51){\footnotesize{$\alpha_{(0,2)}$}}
\put(40,51){\footnotesize{$\alpha_{(1,2)}$}}
\put(55,51){\footnotesize{$\alpha_{(2,2)}$}}
\put(66,51){\footnotesize{$\cdots$}}

\put(26,66){\footnotesize{$\cdots$}}
\put(41,66){\footnotesize{$\cdots$}}
\put(56,66){\footnotesize{$\cdots$}}
\put(66,66){\footnotesize{$\cdots$}}

\psline{->}(35,14)(50,14)
\put(42,10){$\rm{T}_1$}
\psline{->}(10,35)(10,50)
\put(4,42){$\rm{T}_2$}

\put(11,34){\footnotesize{$(0,1)$}}
\put(11,49){\footnotesize{$(0,2)$}}
\put(11,64){\footnotesize{$(0,3)$}}

\put(20,26){\footnotesize{$\beta_{(0,0)}$}}
\put(20,41){\footnotesize{$\beta_{(0,1)}$}}
\put(20,56){\footnotesize{$\beta_{(0,2)}$}}
\put(21,66){\footnotesize{$\vdots$}}

\put(35,26){\footnotesize{$\beta_{(1,0)}$}}
\put(35,41){\footnotesize{$\beta_{(1,1)}$}}
\put(35,56){\footnotesize{$\beta_{(1,2)}$}}
\put(36,66){\footnotesize{$\vdots$}}

\put(50,26){\footnotesize{$\beta_{(2,0)}$}}
\put(50,41){\footnotesize{$\beta_{(2,1)}$}}
\put(50,56){\footnotesize{$\beta_{(2,2)}$}}
\put(51,66){\footnotesize{$\vdots$}}

\put(10,6){(i)}


\put(73,6){(ii)}

\psline{->}(95,14)(110,14)
\put(102,10){$\rm{T}_1$}
\psline{->}(77,35)(77,50)
\put(72,42){$\rm{T}_2$}

\pspolygon*[linecolor=lightgray](95,35)(130,35)(130,70)(95,70)

\psline{->}(80,20)(130,20)
\psline(80,35)(128,35)
\psline(80,50)(128,50)
\psline(80,65)(107,65)

\psline{->}(80,20)(80,70)
\psline(95,20)(95,68)
\psline(110,20)(110,61)
\psline(125,20)(125,58)

\put(75,16){\footnotesize{$(0,0)$}}
\put(91,16){\footnotesize{$(1,0)$}}
\put(106,16){\footnotesize{$(2,0)$}}
\put(121,16){\footnotesize{$(3,0)$}}

\put(85,22){\footnotesize{$\sqrt{\frac{1}{11}}$}}
\put(100,22){\footnotesize{$\sqrt{\frac{1}{2}}$}}
\put(115,22){\footnotesize{$\sqrt{\frac{11}{16}}$}}
\put(126,21){\footnotesize{$\cdots$}}

\put(85,37){\footnotesize{$\sqrt{\frac{1}{8}}$}}
\put(100,37){\footnotesize{$\sqrt{\frac{3}{8}}$}}
\put(115,37){\footnotesize{$\sqrt{\frac{5}{12}}$}}
\put(126,36){\footnotesize{$\cdots$}}

\put(85,52){\footnotesize{$\sqrt{\frac{1}{16}}$}}
\put(100,52){\footnotesize{$\sqrt{\frac{5}{12}}$}}
\put(115,52){\footnotesize{$\sqrt{\frac{9}{20}}$}}
\put(126,51){\footnotesize{$\cdots$}}

\put(85,66){\footnotesize{$\cdots$}}
\put(100,66){\footnotesize{$\cdots$}}

\put(80,27){\footnotesize{$\sqrt{x}$}}
\put(80,41){\footnotesize{$\sqrt{\frac{3}{4}}$}}
\put(80,56){\footnotesize{$\sqrt{\frac{44}{48}}$}}
\put(81,66){\footnotesize{$\vdots$}}

\put(95,27){\footnotesize{$\sqrt{\frac{11}{8}x}$}}
\put(95,41){\footnotesize{$\sqrt{\frac{3}{8}}$}}
\put(95,56){\footnotesize{$\sqrt{\frac{5}{12}}$}}
\put(96,66){\footnotesize{$\vdots$}}

\put(110,27){\footnotesize{$\sqrt{\frac{33}{32}x}$}}
\put(110,41){\footnotesize{$\sqrt{\frac{5}{12}}$}}
\put(110,56){\footnotesize{$\sqrt{\frac{9}{20}}$}}

\put(2,1)
{Figure 1. (i) The weight diagram of a generic 2-variable weighted shift;}
\put(2,-3) {(ii) The weight diagram of the 2-variable weighted shift $\left(T_{1},T_{2}\right)$
given in Theorem \protect \ref{Ans-Lubin}.}

\put(112.5,63){$\left(T_{1},T_{2}\right) |_{\mathcal{M}\cap\mathcal{N}}$}

\end{picture}
\end{center}
\par
\label{Figure Lubin}
\end{figure}

\bigskip

\noindent \S 4. \textbf{A description of the main theorem}. \ For an
arbitrary commuting $2$-variable weighted shift $W_{(\alpha,\beta )}\equiv
(T_{1},T_{2})$, let $\left( T_{1},T_{2}\right) |_{\mathcal{R}}$
denote the restriction of $W_{\left( \alpha ,\beta \right) }$ to $\mathcal{R}%
,$ where $\mathcal{R}$ is a common invariant subspace of $\ell ^{2}(\mathbb{Z}_{+}^{2})$ for
$T_1$ and $T_2$. \
Throughout the paper, we write%
\begin{eqnarray*}
\mathcal{M} &:= &\bigvee \bigl\{e_{(k_{1},k_{2})}\in \ell ^{2}(\mathbb{Z}%
_{+}^{2}):k_{1}\geq 0,k_{2}\geq 1\bigr\}; \\
\mathcal{N} &:= &\bigvee \bigl\{e_{(k_{1},k_{2})}\in \ell ^{2}(\mathbb{Z}%
_{+}^{2}):k_{1}\geq 1,k_{2}\geq 0\bigr\}.
\end{eqnarray*}%
To answer question (\ref{P-Lubin}), we use the $2$-variable weighted shift $%
W_{(\alpha,\beta )}\equiv (T_{1},T_{2})$ with the weight diagram given by Figure
1(ii), where $\alpha _{\left( 0,1\right) }:=\sqrt{\frac{1}{8}}$, {the $0$-th
horizontal slice of $T_{1}$ is a weighted shift }$W_{a}:=\mathrm{shift}%
\left( \alpha _{\left( 0,0\right) },\alpha _{\left( 1,0\right) },\cdots
\right) ${\ whose weight sequence $a\equiv \{a_{n}\}_{n=0}^{\infty }$ is
given by
\begin{equation*}
\begin{tabular}{l}
$a_{n}:=\left\{
\begin{tabular}{ll}
$\sqrt{\frac{1}{11}}$ & $\text{if }n=0$ \\
$\sqrt{\frac{4^{n}+2^{n}+2}{4^{n}+2^{n+1}+8}}$ & $\text{if }n\geq 1$,%
\end{tabular}%
\right. $%
\end{tabular}
\label{alpha1}
\end{equation*}%
and }the $0$-th vertical slice of $T_{2}$ is a weighted shift $W_{b}:=\mathrm{%
shift}\left( \beta _{\left( 0,0\right) },\beta _{\left( 0,1\right) },\cdots
\right) $ whose weight sequence $b\equiv \{b_{n}\}_{n=0}^{\infty }$ is given
by
\begin{equation*}
\begin{tabular}{l}
$b_{n}:=\left\{
\begin{tabular}{ll}
$\sqrt{x }$ \  $\left( {x >0}\right) $ & $\text{if }n=0$ \\
$\sqrt{\frac{10\cdot 2^{2n}+2^{n}+1}{10\cdot 2^{2n}+2^{n+1}+4}}$ & $\text{if
}n\geq 1$.%
\end{tabular}%
\right. $%
\end{tabular}
\label{beta1}
\end{equation*}%
Now, we define $\left( T_{1},T_{2}\right) |_{\mathcal{M}\cap \mathcal{N}}$. \
For both of $0$-th {horizontal and vertical slices of $\left( T_{1},T_{2}\right)
|_{\mathcal{M}\cap \mathcal{N}}$,
we put a
weighted shift $W_{c}$ whose weight sequence $c\equiv
\{c_{n}\}_{n=0}^{\infty }$ is given by
\begin{equation*}
\begin{tabular}{l}
$c_{n}:=\sqrt{\frac{2^{n+1}+1}{2^{n+2}+4}}$ \ ($n\geq 0$):%
\end{tabular}
\label{alpha2}
\end{equation*}%
in other words,%
\begin{equation*}
\begin{tabular}{l}
$W_{c}:=\mathrm{shift}\left( \alpha _{\left( 1,1\right) },\alpha _{\left(
2,1\right) },\cdots \right) =\mathrm{shift}\left( \beta _{\left( 1,1\right)
},\beta _{\left( 1,2\right) },\cdots \right) =\mathrm{shift}\left( \sqrt{%
\frac{3}{8}},\sqrt{\frac{5}{12}},\cdots \right) $.%
\end{tabular}%
\end{equation*}%
\
{In turn, both of the $i$-th horizontal and vertical slices of }$\left( T_{1},T_{2}\right)
|_{\mathcal{M}\cap \mathcal{N}}${\ are defined by a restriction of }$W_{c}${\ to
the subspace $\mathcal{L}_{i}:=\bigvee \{e_{n}:n\geq i\}$, }that is,
\begin{equation*}
\begin{tabular}{l}
$W_{c}|_{{\mathcal{L}_{i}}}=\mathrm{shift}\left( \alpha _{\left( i,i\right)
},\alpha _{\left( i+1,i\right) },\cdots \right) =\mathrm{shift}\left( \beta
_{\left( i,i\right) },\beta _{\left( i,i+1\right) },\cdots \right) $ \\
\\
$=\mathrm{shift}\left( \sqrt{\frac{2^{i+1}+1}{2^{i+2}+4}},\sqrt{\frac{%
2^{i+2}+1}{2^{i+3}+4}},\cdots \right) $.%
\end{tabular}%
\end{equation*}%
{Then, the remaining weights} of $T_{1}$ and $T_2$ are automatically determined by
the commutativity of $T_{1}$ and $T_{2}$. \
Via Berger's Theorem, we can show that}

\begin{itemize}
\item[(a)] $W_{a}$ is subnormal with the $4$-atomic Berger measure
$$
\xi
_{a}:=\frac{3}{4}\delta _{0}+\frac{2}{11}\delta _{\frac{1}{4}}+\frac{1}{22}%
\delta _{\frac{1}{2}}+\frac{1}{44}\delta _{1};
$$

\item[(b)] $W_{b}$ is subnormal with the $4$-atomic Berger measure
$$
\xi _{b}:=\left( 1-\frac{15x}{8}\right) \delta _{0}+x \left( \delta
_{\frac{1}{4}}+\frac{1}{4}\delta _{\frac{1}{2}}+\frac{5}{8}\delta
_{1}\right);
$$

\item[(c)] $W_{c}$ is subnormal with the $2$-atomic Berger measure
$$
\xi _{c}:=\frac{1}{2}\delta _{\frac{1}{4}}+\frac{1}{2}\delta
_{\frac{1}{2}},
$$
\end{itemize}
where $\delta_p$ denotes Dirac measure at $p$. \

\medskip

\noindent
$\{${\it Proof}: For $\ell \geq 1$,
\begin{equation}
\begin{tabular}{l}
$\int s^{\ell }d\xi _{a}(s)=\gamma _{\ell }\left( W_{a}\right) = a_{0}^{2}a_{1}^{2}\cdots a_{\ell
-2}^{2}a_{\ell -1}^{2}$ \\
$=\frac{1}{11}\cdot \frac{4+2+2}{4+2^{2}+8}\cdot \frac{4^{2}+2^{2}+2%
}{4^{2}+2^{3}+8}\cdots \frac{4^{\ell -2}+2^{\ell -2}+2}{%
4^{\ell -2}+2^{\ell -1}+8}\cdot \frac{4^{\ell
-1}+2^{\ell -1}+2}{4^{\ell -1}+2^{\ell }+8}$ \\
$=\frac{1}{11}\cdot \frac{8\left( \frac{1}{4}\right) ^{2}+2\left( \frac{1}{2}%
\right) ^{2}+1}{8\left( \frac{1}{4}\right) +2\left( \frac{1}{2}\right) +1}%
\cdot \frac{8\left( \frac{1}{4}\right) ^{3}+2\left( \frac{1}{2}\right) ^{3}+1%
}{8\left( \frac{1}{4}\right) ^{2}+2\left( \frac{1}{2}\right) ^{2}+1}\cdots
\frac{8\left( \frac{1}{4}\right) ^{\ell -1}+2\left( \frac{1}{2}\right)
^{\ell -1}+1}{8\left( \frac{1}{4}\right) ^{\ell -2}+2\left( \frac{1}{2}%
\right) ^{\ell -2}+1}\cdot \frac{8\left( \frac{1}{4}\right) ^{\ell }+2\left(
\frac{1}{2}\right) ^{\ell }+1}{8\left( \frac{1}{4}\right) ^{\ell -1}+2\left(
\frac{1}{2}\right) ^{\ell -1}+1}$ \\
$=\frac{1}{44}\cdot \left( 8\left( \frac{1}{4}\right) ^{\ell }+2\left( \frac{%
1}{2}\right) ^{\ell }+1\right) =\frac{2}{11}\cdot (\frac{1}{4})^{\ell }+%
\frac{1}{22}\cdot (\frac{1}{2})^{\ell }+\frac{1}{44}$,%
\end{tabular}
\label{cal}
\end{equation}%
giving (a). \
The assertions (b) and (c) follow from the same argument as (a).$\}$

\bigskip

Then, our main theorem follows:

\begin{theorem}
\label{Ans-Lubin} Let $W_{(\alpha,\beta)}\equiv (T_{1},T_{2})$ be given by Figure %
\ref{Figure Lubin}\textrm{(ii)}. Then, we have:

\begin{itemize}
\item[(i)] $T_{1}$ and $T_{2}$ are both subnormal if and only if\ $0<x\leq
\frac{8}{33}$;

\item[(ii)] $(T_{1},T_{2})$ is subnormal if and only if\ $0<x\leq \frac{2}{%
11}$;

\item[(iii)] $T_{1}+T_{2}$ is subnormal if $0<x\leq \frac{2}{11}+\varepsilon $
for some $\varepsilon >0$.
\end{itemize}
\end{theorem}

\bigskip

Consequently, Theorem \ref{Ans-Lubin} proves that there exists a commuting
pair $(T_{1},T_{2})$ of subnormal operators such that $T_{1}+T_{2}$ is
subnormal, but the pair $(T_{1},T_{2})$ is not subnormal; that is, the pair
$(T_{1},T_{2})$ does not admit commuting normal extensions. \ This answers
Lubin's question (\ref{P-Lubin}) in the negative.

\medskip

In Section 2, we give a proof of Theorem \ref{Ans-Lubin}.


\section{\label{main} Proof of Theorem \protect\ref{Ans-Lubin}}

To examine the subnormality of $2$-variable weighted shifts, we need some
definitions.

\medskip

\begin{itemize}
\item[(i)] Let $\mu $ and $\nu $ be two positive measures on a set $X\equiv
\mathbb{R}_{+}$. \ We say that $\mu \leq \nu $ on $X$ if $\mu (E)\leq \nu
(E) $ for each Borel subset $E\subseteq X$; equivalently, $\mu \leq \nu $ if
and only if $\int fd\mu \leq \int fd\nu $ for all $f\in C(X)$ such that $%
f\geq 0$ on $X$, where $C(X)$ denotes the set of all continuous functions on
$X$.

\item[(ii)] Let $\mu$ be a probability measure on $X\times Y\equiv \mathbb{R%
}_{+}\times \mathbb{R}_{+}$ and assume that $\frac{1}{t}\in L^{1}(\mu )$, i.e.,
$\iint \frac{1}{t}\, d\mu(s,t)<\infty$. \
The \textit{extremal measure} $\mu _{ext}$ (which is also a probability
measure) on $X\times Y$ is given by
\begin{equation*}
d\mu _{ext}(s,t):=\frac{1}{t\left\Vert \frac{1}{t}\right\Vert _{L^{1}(\mu )}}%
d\mu (s,t).
\end{equation*}

\item[(iii)] Given a measure $\mu $ on $X\times Y$, the \textit{marginal
measure} $\mu ^{X}$ is given by $\mu ^{X}:=\mu \circ \pi _{X}^{-1}$, where $%
\pi_{X}:X\times Y\rightarrow X$ is the canonical projection onto $X$. \ Thus
$\mu ^{X}(E)=\mu (E\times Y)$ for every $E\subseteq X$.
\end{itemize}

\bigskip

We provide several auxiliary lemmas which are needed for the proof
of Theorem \ref{Ans-Lubin}. \ Recall the subnormal backward
extension of 1-variable weighted shifts (cf. \cite{QHWS}):

If
$\hbox{shift}\,(\alpha
_{1},\alpha _{2},\cdots )$ is subnormal with Berger measure $\xi $, then $%
\hbox{shift}\,(\alpha _{0},\alpha _{1},\alpha _{2},\cdots )$ is subnormal if
and only if
\begin{equation}
\frac{1}{s}\in L^{1}(\xi )\ \ \hbox{and}\ \ \alpha _{0}^{2}\leq \left(
\left\vert \left\vert \frac{1}{s}\right\vert \right\vert _{L^{1}(\xi
)}\right) ^{-1}.  \label{sbe}
\end{equation}

\bigskip

The following lemma is the $2$-variable version of (\ref{sbe}).

\medskip

\begin{lemma}
\label{backext} \textrm{(}\cite[Proposition 3.10]{CuYo1}\textrm{)} \ \textrm{%
(Subnormal backward extension of 2-variable weighted shifts)} \ Assume that $%
W_{(\alpha,\beta )}\equiv (T_{1},T_{2})$ is a commuting pair of subnormal operators
and $(T_{1},T_{2})|_{\mathcal{M}}$ is subnormal with associated Berger
measure $\mu _{\mathcal{M}}$. \
Then, $W_{(\alpha,\beta )}$ is subnormal if and only
if the following conditions hold:

\begin{itemize}
\item[(i)] $\ \frac{1}{t}\in L^{1}(\mu _{\mathcal{M}})$;

\item[(ii)] $\ \beta _{\left( 0,0\right) }^{2}\leq (\left\Vert \frac{1}{t}%
\right\Vert _{L^{1}(\mu _{\mathcal{M}})})^{-1}$;

\item[(iii)] $\ \beta _{\left( 0,0\right) }^{2}\left\Vert \frac{1}{t}%
\right\Vert _{L^{1}(\mu _{\mathcal{M}})}(\mu _{\mathcal{M}})_{ext}^{X}\leq
\xi _{0}$,
\end{itemize}
where $\xi _{0}$ is the Berger measure of $%
\hbox{\rm
shift}\,(\alpha _{\left( 0,0\right) },\alpha _{\left( 1,0\right) },\cdots )$.
In the case when $W_{(\alpha,\beta )}$ is subnormal, the
Berger measure $\mu $ of $W_{(\alpha,\beta )}$ is given by%
\begin{equation*}
d\mu (s,t)=\beta _{\left( 0,0\right) }^{2}\left\Vert \frac{1}{t}\right\Vert
_{L^{1}(\mu _{\mathcal{M}})}d(\mu _{\mathcal{M}})_{ext}(s,t)+\Bigl(d\xi
_{0}(s)-\beta _{\left( 0,0\right) }^{2}\left\Vert \frac{1}{t}\right\Vert
_{L^{1}(\mu _{\mathcal{M}})}d(\mu _{\mathcal{M}})_{ext}^{X}(s)\Bigr)d\delta
_{0}(t).
\end{equation*}
\end{lemma}

On the other hand, we also employ disintegration-of-measure techniques. \ To
do so, we need to review some basic properties on disintegration of measures;
most of the discussion is taken from \cite[VII.2, pp. 317-319]{Con}. \ Let $%
X $ and $Z$ be compact metric spaces and let $\mu $ be a positive regular
Borel measure on $Z$. \ Let $\mathcal{L}^{1}(\mu )$ denote the set of all
Borel functions $f$ on $Z$ such that $\int |f|d\mu <\infty $ and let $%
L^{1}(\mu )$ be the corresponding Lebesgue space of the equivalence classes
of those functions. \ For a Borel mapping $\phi :Z\rightarrow X$, let $\nu $
be the Borel measure $\mu \circ \phi ^{-1}$ on $X$; that is,
\begin{equation}
\nu (\Delta ):=\mu (\phi ^{-1}(\Delta ))
\end{equation}%
for every Borel set $\Delta \subseteq X$. \ If $f\in \mathcal{L}^{1}(\mu )$
then the map $\psi \mapsto \int_{Z}(\psi \circ \phi )f\,d\mu $ defines a
bounded linear functional on $L^{\infty }(\nu )$. \ When restricted to
characteristic functions $\chi _{\Delta }$ in $L^{\infty }(\nu )$, $\Delta
\mapsto \int_{Z}(\chi _{\Delta }\circ \phi )f\,d\mu =\int_{\phi ^{-1}(\Delta
)}f\,d\mu $ is a Borel measure on $X$ which is absolutely continuous with
respect to $\nu $. \
Then, there exists a unique element $E(f)$ in $%
L^{1}(\nu )$ such that $\int_{Z}(\chi _{\Delta }\circ \phi )f\,d\mu
=\int_{X}\chi _{\Delta }E(f)d\nu $ for every Borel set $\Delta $ of $X$. \
Via convergence theorems, one can show that
\begin{equation}\label{disint}
\int_{Z}(\psi \circ \phi )\,f\,d\mu =\int_{X}\psi \,E(f)d\nu
\end{equation}%
for all $\psi \in L^{\infty }(\nu )$. \ This defines a map $E:\mathcal{L}%
^{1}(\mu )\rightarrow L^{1}(\nu )$ called the \textit{expectation operator}.
\ We write $M(Z)$ for the set of all regular Borel measures on $Z$. \
A \textit{disintegration of the measure} $\mu $ with respect to $\phi $ is a
function $x\mapsto \lambda _{x}$ from $X$ to $M(Z)$ such that $\lambda _{x}$
is a probability measure for each $x\in X$ and $E(f)(x)=\int_{Z}f\,d\lambda
_{x}$ a.e. $[\nu ]$ for each $f\in \mathcal{L}^{1}(\mu )$. \ Then we have the
existence and uniqueness of the disintegration of a measure
(cf. \cite[Theorem VII.2.11]{Con}): (i) given a regular Borel measure $\mu $
on a compact metric space $Z$, and a Borel function $\phi $ from $Z$ into a
compact metric space $X$, there is a disintegration $x\mapsto \lambda _{x}$
of $\mu $ with respect to $\phi $; (ii) if $x\mapsto \lambda _{x}^{^{\prime }}$
is another disintegration of $\mu $ with respect to $\phi $, then $\lambda
_{x}=\lambda _{x}^{^{\prime }}$ a.e. $[\nu ]$.

\medskip

The following lemma is useful in the sequel.

\begin{lemma}
\label{Lemma1} If $\mu$ is a positive regular Borel measure defined on $%
Z:= X\times Y\equiv \mathbb{R}_{+}\times \mathbb{R}_{+}$ and $\frac{1}{t}\in L^1(\mu)$, then
\begin{equation*}
\left\Vert \frac{1}{t}\right\Vert _{L^{1}(\mu )}=\left\Vert \frac{1}{t}%
\right\Vert _{L^{1}\left( \mu ^{Y}\right) }\text{,}
\end{equation*}%
where $\mu ^{Y}:=\mu \circ \pi _{Y}^{-1}$ and $\pi _{Y}:Z\rightarrow Y$ is
the canonical projection onto $Y$.
\end{lemma}

\begin{proof}
Put $\phi =\pi _{Y}$ in the preceding argument. \
Then, for the disintegration $t\mapsto \lambda_t$ of the measure $\mu$ with respect to $\phi$,
we know (cf. \cite[Proposition VII.2.10]{Con}) that
$
\hbox{supp}\,(\lambda_t)=\phi^{-1}(t)=X\times \{t\}\subseteq Z.
$
Thus, we may regard $\lambda _{t}$ as a measure on $X$ for each $t\in Y$
and write $d\lambda _{t}(s)$ for $d\lambda _{t}(s,t)$. \
Note that
$$
E(f)(t)=\iint_{X\times Y} f\, d\lambda_t(s,t)=\iint_{X\times\{t\}} f\, d\lambda_t(s,t).
$$
We thus have
\begin{eqnarray*}
\left\Vert \frac{1}{t}\right\Vert _{L^{1}(\mu )}
&=&\iint \frac{1}{t}d\mu (s,t)\\
&=&\int_Y E\left(\frac{1}{t}\right) d\mu^Y(t)\quad\hbox{(by (\ref{disint}) with $\psi\equiv 1$)}                   \\
&=& \int_{Y}\left( \iint_{X\times \{t\}}\frac{1}{t}d\lambda_{t}(s,t)\right) d\mu ^{Y}(t) \\
&=&\int_{Y}\left( \int_{X}\frac{1}{t}d\lambda _{t}(s)\right) d\mu
^{Y}(t)=\int_{Y}\frac{1}{t}d\mu ^{Y}(t) \\
&=&\left\Vert \frac{1}{t}\right\Vert _{L^{1}\left( \mu ^{Y}\right) },
\end{eqnarray*}
which proves the lemma.
\end{proof}

The following is a well-known combinatoric identity, where the first
equality is called the Chu-Vandermonde identity.

\begin{lemma}
\label{chu}
\begin{equation*}
\sum_{k=0}^n \binom{n}{k}^2 = \binom{2n}{n} =\frac{1}{2\pi i} \int_{|z|=1}
\frac{(1+z)^{2n}}{z^{n+1}} \,dz=\frac{1}{\pi} \int_0^4 \frac{s^n}{\sqrt{%
4s-s^2}}ds.
\end{equation*}
\end{lemma}

\begin{proof}
The first equality comes from \cite[(3.66)]{Gou},
the second equality follows from the Cauchy integral formula, and the last equality
follows from a direct calculation.
\end{proof}

\begin{lemma}\label{mn}
If $W_{(\alpha,\beta)}\equiv (T_{1},T_{2})$ is a 2-variable weighted shift given by
Figure \ref{Figure Lubin}\textrm{(ii)},
then $\left( T_{1},T_{2}\right) |_{\mathcal{M\cap N}}$ is subnormal with
Berger measure
\begin{equation}\label{claim_0}
\mu _{\mathcal{M\cap N}}\equiv \frac{1}{2}\delta _{\left(
\frac{1}{4},\frac{1}{4}\right) }+\frac{1}{2}\delta _{\left( \frac{1}{2},%
\frac{1}{2}\right)}.
\end{equation}
\end{lemma}

\begin{proof}
For each $t\in [0,1]$, define
$$
\delta_t(s):=
\begin{cases}
1\ \ &\hbox{if} \ s=t\\
0 &\hbox{otherwise}.
\end{cases}
$$
Then, by the weight diagram of $\left( T_{1},T_{2}\right) |_{\mathcal{M}\cap
\mathcal{N}}$ given in Figure \ref{Figure Lubin}\textrm{(ii), }we can see
that for all $k_{1},k_{2}\geq 0$,%
\begin{eqnarray}
\iint_{\left[ 0,1\right] ^{2}}s^{k_{1}}t^{k_{2}}d\mu _{\mathcal{M\cap N}%
}\left( s,t\right) &=&\int_{0}^{1}t^{k_{1}+k_{2}}d\left( \mu _{\mathcal{%
M\cap N}}\right) ^{Y}\left( t\right)\notag \\
&=&\int_{0}^{1}t^{k_{2}}\left[
  \int_{0}^{1}s^{k_{1}}d\delta _{t}\left( s\right) \right] d\left( \mu _{%
       \mathcal{M\cap N}}\right) ^{Y}\left( t\right)  \notag \\
&=&\iint_{\left[ 0,1\right] ^{2}}s^{k_{1}}t^{k_{2}}d\delta _{t}\left(
     s\right) d\xi_c(t) \notag \\
&=&\iint_{\left[ 0,1\right] ^{2}}s^{k_{1}}t^{k_{2}}d\delta _{t}\left(
     s\right) d \left( \frac{1}{2} \delta _{\frac{1}{4}} +\frac{1}{2%
        }\delta _{\frac{1}{2}} \right) (t)  \label{disintegral} \notag \\
&=&\iint_{\left[ 0,1\right] ^{2}}s^{k_{1}}t^{k_{2}}d\left( \frac{1}{2}%
      \delta _{\left( \frac{1}{4},\frac{1}{4}\right) } +\frac{1}{%
         2}\delta _{\left( \frac{1}{2},\frac{1}{2}\right) } \right) (s,t)
\text{,}  \notag
\end{eqnarray}%
which gives (\ref{claim_0}).
\end{proof}

\bigskip

We are now ready for:

\medskip

\noindent \textbf{Proof of Theorem \ref{Ans-Lubin}}.

\medskip

\noindent (i) \
For $m\ge 0$, let $W_{c}|_{\mathcal{L}_{m}}$ denote the restriction of
$W_{c}$ to $\mathcal{L}_{m}\equiv \bigvee \{e_{(k_1,0)}: k_1\geq m\}$. \
Since $\xi_c=\frac{1}{2}\delta_{\frac{1}{4}}+\frac{1}{2} \delta_{\frac{1}{2}}$, it follows that
for each $m=1,2,\cdots $, $W_{c}|_{\mathcal{L}_{m}}$ is also
subnormal with Berger measure
\begin{equation*}
d(\xi _{c})_{\mathcal{L}_{m}}(s) :=\frac{s^{m}}{\gamma
_{m}\left( W_{c}\right) }d\xi _{c}\left( s\right) =\frac{1}{\gamma
_{m}\left( W_{c}\right) }\left( \frac{1}{2}\left( \frac{1}{4}\right)
^{m}d\delta _{\frac{1}{4}}\left( s\right) +\frac{1}{2}\left( \frac{1}{2}%
\right) ^{m}d\delta _{\frac{1}{2}}\left( s\right) \right) \text{.}
\end{equation*}%
Note%
\begin{equation*}
\left\Vert \frac{1}{s}\right\Vert _{L^{1}((\xi _{c})_{\mathcal{L}_{m}})}=%
\frac{2\left( \frac{1}{4}\right) ^{m}+\left( \frac{1}{2}\right) ^{m}%
}{\gamma _{m}\left( W_{c}\right) }\quad \text{and}\quad \alpha
_{(0, m+1)}^{2}=\frac{\gamma _{m}\left( W_{c}\right) }{8\cdot \gamma
_{m}\left( W_{b}|_{\mathcal{L}_{1}}\right) }\, ,
\end{equation*}%
where the Berger measure of $W_{b}|_{\mathcal{L}_{1}}$ is $(\xi_{b})_{%
\mathcal{L}_{1}}:=\frac{1}{4}\delta _{\frac{1}{4}}+\frac{1}{8}\delta _{\frac{%
1}{2}}+\frac{5}{8}\delta _{1}$. \
Since the $m$-th horizontal slice of $\left( T_{1},T_{2}\right) |_{\mathcal{M\cap N}}$
is a restriction of $W_c$ to $\mathcal{L}_m$, it follows from
(\ref{sbe}) that $T_{1}$ is subnormal
if and only if $\alpha _{(0,m+1)}^{2}\leq \left\Vert \frac{1}{s}%
\right\Vert _{L^{1}((\xi _{c})_{\mathcal{L}_{m}})}^{-1}$ $\left( \text{%
all} \ m\geq 0\right) $. \ Since
$$
2\left( \frac{1}{4}\right)^{m} +\left(\frac{1}{2}\right)^{m}\leq
8\cdot \gamma _{m}\left( W_{b}|_{\mathcal{L}_1}\right) = 8\int s^m
d(\xi_b)_{\mathcal{L}_1} \hskip -0.1cm (s)= 2\left( \frac{1}{4}\right) ^{m}+\left(
\frac{1}{2}\right) ^{m}+5 \quad (\hbox{all} \ m\geq 0),
$$
it follows at once that $T_{1}$ is subnormal.

Similarly, if $n\geq 0$, then
\begin{equation}
\left\Vert \frac{1}{t}\right\Vert _{L^{1}((\xi _{c})_{\mathcal{L}_{n}})}=%
\frac{2\left( \frac{1}{4}\right) ^{n}+\left( \frac{1}{2}\right) ^{n}%
}{\gamma _{n}\left( W_{c}\right) }\quad \text{and}\quad \beta
_{(n+1,0)}^{2}=\frac{11x\cdot \gamma _{n}\left( W_{c}\right) }{%
8\cdot \gamma _{n}\left( W_{a}|_{\mathcal{L}_{1}}\right) } \, , \label{sde}
\end{equation}%
where the Berger measure of $W_{a}|_{\mathcal{L}_{1}}$ is $(\xi _{a})_{%
\mathcal{L}_{1}}:=\frac{1}{2}\delta _{\frac{1}{4}}+\frac{1}{4}\delta _{\frac{%
1}{2}}+\frac{1}{4}\delta _{1}$. \ Since $T_{2}$ is subnormal if and only if $%
\beta _{(n+1,0)}^{2}\leq \left\Vert \frac{1}{t}\right\Vert _{L^{1}((\xi
_{c})_{\mathcal{L}_{n}})}^{-1}$ $\left( \text{all }n\geq 0\right) $,
a direct calculation together with (\ref{sbe}) and (\ref{sde}) shows that
\begin{equation*}
T_{2}\ \hbox{is subnormal}\ \Longleftrightarrow \ x\leq \frac{8\left( \frac{1%
}{2}\left( \frac{1}{4}\right)^{n}+\frac{1}{4}\left( \frac{1}{2}\right)
^{n}+\frac{1}{4}\right) }{11\left( 2\left( \frac{1}{4}\right)
^{n}+\left( \frac{1}{2}\right) ^{n}\right) }\
\hbox{(all $n\ge
0$)}\ \Longleftrightarrow \ x\leq \frac{8}{33}\text{,}
\end{equation*}%
where the second implication follows from the observation that the fractional
function of the second term is increasing on $n\geq 0$. \
This proves (i).

\medskip

\noindent (ii) \ We first claim that
\begin{equation}
\left( T_{1},T_{2}\right) |_{\mathcal{M}}\
\hbox{is subnormal with Berger
measure}\ \mu _{\mathcal{M}}\equiv
\frac{1}{4}\delta _{\left( \frac{1}{4},\frac{1}{4}\right) }+\frac{1}{8}%
\delta _{\left( \frac{1}{2},\frac{1}{2}\right) }+\frac{5}{8}\delta _{\left( 0,1\right) }.  \label{claim_1}
\end{equation}%
For (\ref{claim_1}), we first observe that by Lemma \ref{Lemma1}, $\left\Vert
\frac{1}{s}\right\Vert _{L^{1}(\mu _{\mathcal{M\cap N}})}=\left\Vert \frac{1%
}{s}\right\Vert _{L^{1}\left( \left( \mu _{\mathcal{M\cap N}}\right)
^{X}\right) }=3$ since $\left( \mu _{\mathcal{M\cap N}}\right) ^{X}=\frac{1%
}{2}\delta _{\frac{1}{4}}+\frac{1}{2}\delta _{\frac{1}{2}}$ (by Lemma \ref{mn}). \
We thus have
\begin{equation}
(\mu _{\mathcal{M\cap N}})_{ext}^{Y}=\left( \left\Vert \frac{1}{s}%
\right\Vert _{L^{1}(\mu _{\mathcal{M\cap N}})}^{-1}  \hskip -0.3cm  \frac{\mu _{\mathcal{%
M\cap N}}}{s}\right) ^{Y}=\frac{2}{3}\delta
_{\frac{1}{4}}+\frac{1}{3}\delta _{\frac{1}{2}}\text{.}  \label{2}
\end{equation}%
Hence, by Lemma \ref{backext}(iii), $\left( T_{1},T_{2}\right) |_{\mathcal{M}%
}$ is subnormal if and only if
\begin{equation*}
\alpha _{(0,1)}^{2}\left\Vert \frac{1}{s}\right\Vert _{L^{1}(\mu _{%
\mathcal{M\cap N}})} \hskip -0.3cm    (\mu _{\mathcal{M\cap N}})_{ext}^{Y} \leq (\xi_{b})_{\mathcal{L}_{1}} \
\Longleftrightarrow \ \frac{1}{4}\delta _{\frac{1}{4}}+%
\frac{1}{8}\delta _{\frac{1}{2}}\leq \frac{1}{4}\delta _{\frac{1}{4}}+\frac{1%
}{8}\delta _{\frac{1}{2}}+\frac{5}{8}\delta _{1},
\end{equation*}%
which is always true. \ Therefore, $\left( T_{1},T_{2}\right) |_{\mathcal{M}%
} $ is always subnormal. \ By Lemma \ref{backext} and (\ref{2}), we get the
desired Berger measure of $\left( T_{1},T_{2}\right) |_{\mathcal{M}}$:
$$
\begin{aligned}
d \mu_{\mathcal M}(s,t)
&=\alpha_{(0,1)}^2 \left\Vert \frac{1}{s}\right\Vert_{L^1(\mu_{\mathcal{M\cap N}})} \hskip -0.3cm d(\mu_{\mathcal{M\cap N}})_{ext}(s,t)\\
&\qquad\qquad +\Biggl(d(\xi_b)_{\mathcal L_1}(t)-\alpha_{(0,1)}^2\left\Vert \frac{1}{s}\right\Vert_{L^1(\mu_{\mathcal{M\cap N}})}
   \hskip -0.3cm d(\mu_{\mathcal{M\cap N}})_{ext}^Y(t)\Biggr) d \delta_0(s)\\
&=\frac{3}{8} \left( \frac{2}{3} d \delta_{(\frac{1}{4},\frac{1}{4})} (s,t)
    + \frac{1}{3} d \delta_{(\frac{1}{2},\frac{1}{2})}(s,t)\right)
      + \frac{5}{8} d \delta_1(t) d \delta_0(s)\\
&= \frac{1}{4}  d \delta_{(\frac{1}{4},\frac{1}{4})} (s,t) +\frac{1}{8}d \delta_{(\frac{1}{2},\frac{1}{2})} (s,t)
   +\frac{5}{8} d \delta_{(0,1)} (s,t),
\end{aligned}
$$
which gives
\begin{equation*}
\mu _{\mathcal{M}} = \frac{1}{4}%
\delta _{\left( \frac{1}{4},\frac{1}{4}\right) }+\frac{1}{8}\delta _{\left(
\frac{1}{2},\frac{1}{2}\right) }
+\frac{5}{8}\delta _{\left( 0,1\right) }
\text{.}
\end{equation*}%
We next claim that
\begin{equation*}
\left( T_{1},T_{2}\right) |_{\mathcal{N}}\ \hbox{is subnormal}\
\Longleftrightarrow \ 0<x\leq \frac{2}{11}.  \label{claim2}
\end{equation*}%
By Lemma \ref{Lemma1}, we note that $\left\Vert \frac{1}{t}\right\Vert
_{L^{1}(\mu _{\mathcal{M\cap N}})}=\left\Vert \frac{1}{t}\right\Vert
_{L^{1}\left( \left( \mu _{\mathcal{M\cap N}}\right) ^{Y}\right) }=3$ and $%
(\mu _{\mathcal{M\cap N}})_{ext}^{X} (s)=\frac{2}{3}\delta _{\frac{1}{4}}+\frac{1%
}{3}\delta _{\frac{1}{2}}$. \ Thus, by Lemma \ref{backext}(iii), $\left(
T_{1},T_{2}\right) |_{\mathcal{N}}$ is subnormal if and only if
\begin{equation*}
\begin{tabular}{l}
$\beta _{\left( 1,0\right) }^{2}\left\Vert \frac{1}{t}\right\Vert
_{L^{1} \left( \mu _{\mathcal{M\cap N}}\right) }(\mu _{%
\mathcal{M\cap N}})_{ext}^{X}  \leq (\xi_{a})_{\mathcal{L}_{1}}$ \\
$\Longleftrightarrow
\frac{11x}{8}\cdot 3 \cdot \left(\frac{2}{3} \delta_{\frac{1}{4}}+\frac{1}{3}
\delta _{\frac{1}{2}} \right) \leq \frac{1}{2}\delta _{\frac{1}{4}}+\frac{1}{4}\delta
_{\frac{1}{2}}+\frac{1}{4}\delta _{1}\Longleftrightarrow x\leq \frac{2}{11}$.%
\end{tabular}%
\end{equation*}%
We now claim that
\begin{equation}
\left( T_{1},T_{2}\right) \ \hbox{is subnormal}\ \Longleftrightarrow \
0<x\leq \frac{2}{11}.  \label{claim_3}
\end{equation}%
Towards (\ref{claim_3}), observe that the commutativity of $T_{1}$ and $%
T_{2} $ comes directly from Figure 1(ii). \ By the proof of (i) just given
above, we know that $T_{1}$ is always subnormal and
$$
T_{2}\ \hbox{is subnormal} \ \Longleftrightarrow \
0<x\leq \frac{8}{33}.
$$
By Lemma \ref{Lemma1}, we have
$\left\Vert \frac{1}{t}\right\Vert _{L^{1}(\mu _{\mathcal{M}})}=\left\Vert
\frac{1}{t}\right\Vert _{L^{1}\left( \left( \mu _{\mathcal{M}}\right)
^{Y}\right) }=\frac{15}{8}$ (since $\mu _{\mathcal{M}}^{Y}=\frac{1}{4}%
\delta _{\frac{1}{4}}+\frac{1}{8}\delta _{\frac{1}{2}}+\frac{5}{8}\delta
_{1}$) and
$$
(\mu _{\mathcal{M}})_{ext}^{X}=\left( \left\Vert \frac{1}{t}%
\right\Vert _{L^{1}(\mu _{\mathcal{M}})}^{-1}\frac{\mu _{\mathcal{M}}}{t}%
\right) ^{X} =\frac{1}{3}\delta _{0}+
\frac{8}{15}\delta _{\frac{1}{4}}+\frac{2%
}{15}\delta _{\frac{1}{2}}.
$$
Hence, by Lemma \ref{backext}(iii), $\left(
T_{1},T_{2}\right) $ is subnormal if and only if
\begin{equation*}
\begin{tabular}{l}
$\beta _{00}^{2}\left\Vert \frac{1}{t}\right\Vert _{L^{1}\left( \mu _{%
\mathcal{M}}\right) }(\mu _{\mathcal{M}})_{ext}^{X}\leq \xi _{a}$
\\
$\Longleftrightarrow x\left( \frac{5}{8}\delta _{0}+\delta _{\frac{1}{4}}+%
\frac{1}{4}\delta _{\frac{1}{2}}\right) \leq \frac{3}{4}\delta _{0}+\frac{2}{%
11}\delta _{\frac{1}{4}}+\frac{1}{22}\delta _{\frac{1}{2}}+\frac{1}{44}%
\delta _{1}\Longleftrightarrow x\leq \frac{2}{11},$
\end{tabular}%
\end{equation*}
which proves (ii).

\

\medskip

\noindent (iii) \ For the subnormality of $T_{1}+T_{2}$, we shall use Agler's
criterion for subnormality in \cite{Agl}, which states that a contraction $%
S\in \mathcal{B(H)}$ is subnormal if and only if $\sum\limits_{\ell
=0}^{n}(-1)^{\ell }{\binom{n}{l}}\left\Vert S^{\ell }x\right\Vert ^{2}\geq 0$
for all $n\geq 1$ and all $x\in \mathcal{H}$. \ Since $(T_{1},T_{2})|_{%
\mathcal{M}}$ is subnormal, it is enough to consider Agler's criterion at $%
\left\{ e_{\left( k,0\right) }\right\} _{k=0}^{\infty }$: indeed, if $%
x=\sum_{k}a_{k}e_{(k,0)}$, then
\begin{equation*}
\left\Vert\left(\frac{T_{1}+T_{2}}{2}\right)^{\ell }x\right\Vert^{2}
=\sum_{k}|a_{k}|^{2}
\left\Vert\left(\frac{T_{1}+T_{2}}{2}\right)^{\ell }e_{(k,0)}\right\Vert^{2},
\end{equation*}%
and hence
\begin{equation*}
\sum_{\ell =0}^{n}(-1)^{\ell }{\binom{n}{l}} \left\Vert\left(\frac{T_{1}+T_{2}}{2}\right)^{\ell
}x\right\Vert^{2}=\sum_{k}|a_{k}|^{2}\sum_{\ell =0}^{n}(-1)^{\ell }{\binom{n}{l}}
\left\Vert\left(
\frac{T_{1}+T_{2}}{2}\right)^{\ell }e_{(k,0)}\right\Vert^{2},
\end{equation*}%
which gives
\begin{equation*}
\frac{T_{1}+T_{2}}{2}\ \text{is subnormal}\ \Longleftrightarrow \
P_{n}\left( k,0\right) :=\sum\limits_{\ell =0}^{n}(-1)^{\ell }{\binom{n}{%
\ell }}\left\Vert \left( \frac{T_{1}+T_{2}}{2}\right) ^{\ell }e_{\left(
k,0\right) }\right\Vert ^{2}\geq 0\quad \left( \text{all }n\geq 1\right)
\text{.}
\end{equation*}%
Hence, we see that $T_{1}+T_{2}$ is subnormal if and only if $\inf \Bigl\{
P_{n}\left( k,0\right) :\ n\in \mathbb{Z}_{+}\Bigr\} \geq 0$ for all $k\geq
0$. \ For $\ell \geq 1$, we observe
\begin{equation*}
\left( \frac{T_{1}+T_{2}}{2}\right) ^{\ell }=2^{-\ell }\left( T_{1}^{\ell
}+T_{2}^{\ell }+\sum\limits_{i=1}^{\ell -1}{\binom{\ell }{i}}T_{1}^{\ell
-i}T_{2}^{i}\right) \text{.}
\end{equation*}%
First of all, we suppose $k\geq 1$. \ We then have
\begin{equation*}
\begin{tabular}{l}
$P_{n}\left( k,0\right) =\sum\limits_{\ell =0}^{n}(-1)^{\ell }{\binom{n}{%
\ell }}\left\Vert \left( \frac{T_{1}+T_{2}}{2}\right) ^{\ell }e_{\left(
k,0\right) }\right\Vert ^{2}$ \\
\\
$=1+\sum\limits_{\ell =1}^{n}\left( -1\right) ^{\ell }{\binom{n}{\ell }}%
2^{-2\ell }\left( \frac{\gamma _{k+\ell }(\xi _{a})}{\gamma _{k}(\xi _{a})}+%
\frac{x}{8}\frac{\gamma _{k+\ell -2}\bigl((\mu _{\mathcal{M\cap N}})^{X}\bigr)}{\gamma
_{k}(\xi _{a})}+\sum\limits_{i=1}^{\ell -1}\binom{\ell }{i}^{2}\frac{x}{8}%
\frac{\gamma _{k+\ell -2}\bigl((\mu _{\mathcal{M\cap N}})^{X}\bigr)}{\gamma _{k}(\xi
_{a})}\right) $ \\
\\
$=1+\sum\limits_{\ell =1}^{n}\left( -1\right) ^{\ell }{\binom{n}{\ell }}%
2^{-2\ell }\left( \frac{\gamma _{k+\ell }(\xi _{a})}{\gamma _{k}(\xi _{a})}+%
\frac{x}{8}\frac{\gamma _{k+\ell -2}\bigl((\mu _{\mathcal{M\cap N}})^{X}\bigr)}{\gamma
_{k}(\xi _{a})}\left( \sum\limits_{i=1}^{\ell -1}\binom{\ell }{i}%
^{2}+1\right) \right) $,%
\end{tabular}
\label{equat7}
\end{equation*}%
where $\gamma _{\ell }(\xi _{a})$ and $\gamma _{\ell }\bigl((\mu _{\mathcal{M%
}\cap \mathcal{N}})^{X}\bigr)$ denote the $\ell $-th moments of $\hbox{shift}%
\,\left( \alpha _{\left( 0,0\right) },\alpha _{\left( 1,0\right) },\cdots
\right) $ and the $0$-th horizontal slice of $(T_{1},T_{2})|_{\mathcal{M}\cap \mathcal{N}}$, respectively.
\ Note that
\begin{equation*}
\begin{cases}
\gamma _{\ell }(\xi _{a})=\frac{2}{11}\left( \frac{1}{4}\right) ^{\ell }+%
\frac{1}{22}\left( \frac{1}{2}\right) ^{\ell }+\frac{1}{44} \\
\gamma _{\ell }\left( (\mu _{\mathcal{M\cap N}})^{X}\right) =\frac{1}{2}\left( \frac{1%
}{4}\right) ^{\ell }+\frac{1}{2}\left( \frac{1}{2}\right) ^{\ell }.%
\end{cases}
\label{gamma}
\end{equation*}%
We thus have
\begin{equation*}
\begin{tabular}{l}
$P_{n}(k,0)=1+\frac{1}{\gamma _{k}(\xi _{a})}\left( \sum\limits_{\ell
=1}^{n}\left( -1\right) ^{\ell }{\binom{n}{\ell }}2^{-2\ell }\left( \frac{2}{%
11}\left( \frac{1}{4}\right) ^{k+\ell }+\frac{1}{22}\left( \frac{1}{2}%
\right) ^{k+\ell }+\frac{1}{44}\right) \right. $ \\
$\ \ \ \ \ \ \ \ \ \ \ \ \ \ \ \ \ \ \ \ \ \ \ \ \ \ \ \ \ +\left. \frac{x}{8%
}\sum\limits_{\ell =1}^{n}\left( -1\right) ^{\ell }{\binom{n}{\ell }}%
2^{-2\ell }\left( \frac{1}{2}\left( \frac{1}{4}\right) ^{k+\ell -2}+\frac{1}{%
2}\left( \frac{1}{2}\right) ^{k+\ell -2}\right) \left(
\sum\limits_{i=1}^{\ell -1}\binom{\ell }{i}^{2}+1\right) \right) .$%
\end{tabular}
\label{2.8}
\end{equation*}%
Observe that
\begin{equation}
\sum_{i=1}^{\ell -1}{\binom{\ell }{i}}^{2}+1={\binom{2\ell }{\ell }}-1\quad %
\hbox{(by Lemma \ref{chu})}  \label{CV}
\end{equation}%
and
\begin{equation}
\sum\limits_{\ell =1}^{n}\left( -1\right) ^{\ell }{\binom{n}{\ell }}c^{\ell
}=\left( 1-c\right) ^{n}-1\ \ (0<c<1).  \label{binom}
\end{equation}%
By (\ref{CV}) and (\ref{binom}), $P_{n}\left( k,0\right) $ can be written as
\begin{equation}
\begin{tabular}{l}
$P_{n}\left( k,0\right) =1+\frac{1}{\gamma _{k}(\xi _{a})}\left( \left(
\frac{2}{11}-x\right) \left( \frac{1}{4}\right) ^{k}\left( \left( \frac{15}{%
16}\right) ^{n}-1\right) +\left( \frac{1}{22}-\frac{x}{4}\right) \left(
\frac{1}{2}\right) ^{k}\left( \left( \frac{7}{8}\right) ^{n}-1\right) +\frac{%
1}{44}\left( \left( \frac{3}{4}\right) ^{n}-1\right) \right. $ \\
$\ \ \ \ \ \ \ \ \ \ \ \ \ \left. +x\left( \frac{1}{4}\right)
^{k}\sum\limits_{\ell =1}^{n}(-1)^{\ell }{\binom{n}{\ell }}{\binom{2\ell }{%
\ell }}\left( \frac{1}{16}\right) ^{\ell }+\frac{x}{4}\left( \frac{1}{2}%
\right) ^{k}\sum\limits_{\ell =1}^{n}(-1)^{\ell }{\binom{n}{\ell }}{\binom{%
2\ell }{\ell }}\left( \frac{1}{8}\right) ^{\ell }\right) .$%
\end{tabular}
\label{**}
\end{equation}%
Now, we should resolve the last two terms of (\ref{**}). \ To do so, we
consider the following weighted shift%
\begin{equation*}
W_{S}:=\mathrm{shift}\left( \frac{\left\Vert Se_{\left( 0,0\right)
}\right\Vert }{\left\Vert e_{\left( 0,0\right) }\right\Vert },\frac{%
\left\Vert S^{2}e_{\left( 0,0\right) }\right\Vert }{\left\Vert Se_{\left(
0,0\right) }\right\Vert },\frac{\left\Vert S^{3}e_{\left( 0,0\right)
}\right\Vert }{\left\Vert S^{2}e_{\left( 0,0\right) }\right\Vert },\cdots
\right) \text{,}
\end{equation*}%
where $S:=U_{+}\otimes I+I\otimes U_{+}\in \mathcal{B}(\ell ^{2}(\mathbb{Z}%
_{+}^{2}))$ (where $U_{+}\equiv \mathrm{shift}(1,1,\cdots )$ is the unilateral shift), which is
subnormal. \ By Lambert's Theorem in \cite{Lam} and Berger's Theorem, we can
see that $W_{S}$ is subnormal and%
\begin{equation}
\int_{0}^{4}s^{\ell }d\mu \left( s\right) =\gamma _{\ell }\left(
W_{S}\right) =\left\Vert S^{\ell }e_{\left( 0,0\right) }\right\Vert ^{2}=
\sum_{k=0}^{\ell} \binom{\ell}{k}^2={%
\binom{2\ell }{\ell }}\,,  \label{equ2}
\end{equation}%
where $\mu $ is the Berger measure corresponding to the subnormal weighted
shift $W_{S}$. \ We thus have%
\begin{equation}
\begin{aligned} \sum\limits_{\ell =1}^{n}\left( -1\right) ^{\ell
}{\binom{n}{\ell }}{\binom{2\ell }{\ell }} \left(\frac{1}{16}\right)^\ell
&=\sum\limits_{\ell =1}^{n}\left( -1\right) ^{\ell }{\binom{n}{\ell }}\left(
\int_{0}^{4}s^{\ell }d\mu \left( s\right) \right)
\left(\frac{1}{16}\right)^\ell \quad \text{(by (\ref{equ2}))}\\
&=\int_{0}^{4}\left( \sum\limits_{\ell =0}^{n}{\binom{n}{\ell }}\left(
-1\right) ^{\ell }\left( \frac{s}{16}\right) ^{\ell }\right) d\mu \left(
s\right) -1\\ &=\int_{0}^{4}\left( 1-\frac{s}{16}\right) ^{n}d\mu \left(
s\right) -1  \quad \text{(by (\ref{binom}))} \end{aligned}  \label{equ6}
\end{equation}%
and similarly,
\begin{equation}
\sum\limits_{\ell =1}^{n}\left( -1\right) ^{\ell }{\binom{n}{\ell }}{\binom{%
2\ell }{\ell }}\left( \frac{1}{8}\right) ^{\ell }=\int_{0}^{4}\left( 1-\frac{%
s}{8}\right) ^{n}d\mu \left( s\right) -1\text{.}  \label{equ7}
\end{equation}%
By (\ref{equ6}) and (\ref{equ7}), (\ref{**}) can be written as
\begin{equation*}
\begin{tabular}{l}
$P_{n}\left( k,0\right) =1+\frac{1}{\gamma _{k}(\xi _{a})}\left( \left(
\frac{2}{11}-x\right) \left( \frac{1}{4}\right) ^{k}\left( \left( \frac{15}{%
16}\right) ^{n}-1\right) +\left( \frac{1}{22}-\frac{x}{4}\right) \left(
\frac{1}{2}\right) ^{k}\left( \left( \frac{7}{8}\right) ^{n}-1\right) +\frac{
1}{44}\left( \left( \frac{3}{4}\right) ^{n}-1\right) \right. $ \\
$\ \ \ \ \ \ \ \ \ \ \ \ \ \left. +\ x\left( \frac{1}{4}\right) ^{k}\left(
\int_{0}^{4}\left( 1-\frac{s}{16}\right) ^{n}d\mu \left( s\right) -1\right) +%
\frac{x}{4}\left( \frac{1}{2}\right) ^{k}\left( \int_{0}^{4}\left( 1-\frac{s%
}{8}\right) ^{n}d\mu \left( s\right) -1\right) \right) .$%
\end{tabular}
\label{equa7}
\end{equation*}%
Since $\gamma _{k}(\xi _{a})=\frac{2}{11}\left( \frac{1}{4}\right) ^{k}+%
\frac{1}{22}\left( \frac{1}{2}\right) ^{k}+\frac{1}{44}$, it follows that
\begin{equation}
\begin{tabular}{l}
$P_{n}\left( k,0\right) =\frac{1}{\gamma _{k}(\xi _{a})}\left( \left( \frac{2%
}{11}-x\right) \left( \frac{1}{4}\right) ^{k}\left( \frac{15}{16}\right)
^{n}+\left( \frac{1}{22}-\frac{x}{4}\right) \left( \frac{1}{2}\right)
^{k}\left( \frac{7}{8}\right) ^{n}+\frac{1}{44}\left( \frac{3}{4}\right)
^{n}\right. $ \\
$\ \ \ \ \ \ \ \ \ \ \ \ \ \left. +\ x\left( \frac{1}{4}\right)
^{k}\int_{0}^{4}\left( 1-\frac{s}{16}\right) ^{n}d\mu \left( s\right) +\frac{%
x}{4}\left( \frac{1}{2}\right) ^{k}\int_{0}^{4}\left( 1-\frac{s}{8}\right)
^{n}d\mu \left( s\right) \right) ,$%
\end{tabular}
\label{77}
\end{equation}%
which implies that
\begin{equation}
\begin{tabular}{l}
$P_{n}\left( k,0\right) \gamma _{k}(\xi _{a})=\frac{2}{11}\left( \frac{1}{4}%
\right) ^{k}\left( \frac{15}{16}\right) ^{n}+\frac{1}{22}\left( \frac{1}{2}%
\right) ^{k}\left( \frac{7}{8}\right) ^{n}+\frac{1}{44}\left( \frac{3}{4}%
\right) ^{n}$ \\
$\ \ \ \ \ \ \ \ \ \ \ \ \ +\ x\left( \frac{1}{4}\right) ^{k}\left(
\int_{0}^{4}\left( 1-\frac{s}{16}\right) ^{n}d\mu \left( s\right) -\left(
\frac{15}{16}\right) ^{n}\right) +\frac{x}{4}\left( \frac{1}{2}\right)
^{k}\left( \int_{0}^{4}\left( 1-\frac{s}{8}\right) ^{n}d\mu \left( s\right)
-\left( \frac{7}{8}\right) ^{n}\right) .$%
\end{tabular}
\label{equ777}
\end{equation}%
Observe that by Lemma \ref{chu},
\begin{equation*}
d\mu (s)=\frac{1}{\pi }\frac{ds}{\sqrt{4s-s^{2}}}.
\end{equation*}%
We thus have
\begin{equation*}
\begin{aligned} \frac{\int_0^4 \left(1-\frac{s}{16}\right)^n
d\mu(s)}{\left(\frac{15}{16}\right)^n} &= \frac{1}{\pi} \int_0^4
\left(\frac{16-s}{15}\right)^n\frac{ds}{\sqrt{4s-s^2}}\\ &\geq
\frac{1}{\pi}\int_{\frac{1}{3}}^{\frac{1}{2}}\left(\frac{16-s}{15}\right)^n%
\frac{ds}{\sqrt{4s-s^2}}\\
&=\frac{1}{6\pi}\left(\frac{16-s_0}{15}\right)^n\frac{1}{\sqrt{4s_0-s_0^2}}\
\ \ \hbox{(for some $s_0$ with $\frac{1}{3}<s_0<\frac{1}{2}$)} ,
\end{aligned}
\end{equation*}%
which tends to $\infty $ as $n\rightarrow \infty $ and similarly,
\begin{equation*}
\frac{\int_{0}^{4}\left( 1-\frac{s}{8}\right) ^{n}d\mu (s)}{\left( \frac{7}{8%
}\right) ^{n}}\rightarrow \infty \ \ \hbox{as}\ n\rightarrow \infty .
\end{equation*}%
This implies that by (\ref{equ777}), there exists $n_{0}\in \mathbb{Z}_{+}$
such that
\begin{equation}
P_{n}(k,0)\geq 0\ \ \hbox{if}\ n>n_{0}.  \label{777-1}
\end{equation}%
Now, suppose
\begin{equation*}
\begin{aligned} \varepsilon_1 &:= \min_{1\le n\le n_0} \int_{0}^{4}\left(
1-\frac{s}{16}\right) ^{n}d\mu \left( s\right) = \min_{1\le n\le n_0}
\frac{1}{\pi} \int_{0}^{4}\left( 1-\frac{s}{16}\right) ^{n}
\frac{ds}{\sqrt{4s-s^2}};\\ \varepsilon_2 &:= \min_{1\le n\le n_0}
\int_{0}^{4}\left( 1-\frac{s}{8}\right) ^{n}d\mu \left( s\right) =
\min_{1\le n\le n_0} \frac{1}{\pi} \int_{0}^{4}\left( 1-\frac{s}{8}\right)
^{n}\frac{ds}{\sqrt{4s-s^2}} \end{aligned}
\end{equation*}%
and put $\varepsilon :=\min \{\varepsilon _{1},\varepsilon _{2}\}$. \
Obviously, $\varepsilon >0$. Thus, by (\ref{77}),
\begin{equation*}
\begin{tabular}{l}
$P_{n}\left( k,0\right) \geq \frac{1}{\gamma _{k}(\xi _{a})}\left( \left(
\frac{2}{11}-x+\varepsilon \right) \left( \frac{1}{4}\right) ^{k}\left(
\frac{15}{16}\right) ^{n}+\left( \frac{1}{22}-\frac{x}{4}+\frac{\varepsilon
}{4}\right) \left( \frac{1}{2}\right) ^{k}\left( \frac{7}{8}\right) ^{n}+%
\frac{1}{44}\left( \frac{3}{4}\right) ^{n}\right) ,$%
\end{tabular}
\label{7771}
\end{equation*}%
which implies that
\begin{equation}
P_{n}(k,0)\geq 0\ \ \hbox{($1\le n\le n_0)$}\ \
\hbox{whenever $0< x
\le \frac{2}{11}+\varepsilon$}.  \label{777-2}
\end{equation}%
By (\ref{777-1}) and (\ref{777-2}), we can conclude that for each $k\geq 1$,
$P_{n}(k,0)\geq 0$ for all $n\in \mathbb{Z}_{+}$ if $0<x\leq \frac{2}{11}%
+\varepsilon $ (some $\varepsilon >0$).

If instead $k=0$ then the same argument shows that
\begin{equation*}
\begin{tabular}{l}
$P_{n}(0,0)=\left( \frac{3}{4}-\frac{5x}{8}\right) +\left( \frac{2}{11}%
-x\right) \left( \frac{15}{16}\right) ^{n}+\left( \frac{1}{22}-\frac{x}{4}%
\right) \left( \frac{7}{8}\right) ^{n}+\left( \frac{1}{44}+\frac{5x}{8}%
\right) \left( \frac{3}{4}\right) ^{n}$ \\
$\ \ \ \ \ \ \ \ \ \ \ \ \ +\ x\int_{0}^{4}\left( 1-\frac{s}{16}\right)
^{n}d\mu \left( s\right) +\frac{x}{4}\int_{0}^{4}\left( 1-\frac{s}{8}\right)
^{n}d\mu \left( s\right) ,$%
\end{tabular}
\label{777-3}
\end{equation*}%
which also implies that
\begin{equation*}
P_{n}(0,0)\geq 0\ \ \hbox{(all $n\in\mathbb{Z}_+$)}\ \
\hbox{whenever $0<
\epsilon \le \frac{2}{11}+\varepsilon$}.  \label{777-4}
\end{equation*}%
Therefore, we can conclude that $T_{1}+T_{2}$ is subnormal if $0<x\leq \frac{%
2}{11}+\varepsilon $ (some $\varepsilon >0$). \ This proves the theorem. \hfill
\qed

\medskip

\begin{remark}
Our $2$-variable weighted shift in Theorem \ref{Ans-Lubin} has $4$-atomic
Berger measures in {the $0$-th horizontal and vertical slices of }$\left( {%
T_{1}},{T_{2}}\right) $. \ However, if we take $3$-atomic Berger
measures in {the $0$-th horizontal and vertical slices of }$\left(
{T_{1}},{T_{2}}\right) $, then our extensively numerous trials
resisted resolution for finding a gap between the subnormality of
$T_{1}+T_{2}$ and the subnormality of $\left( {T_{1}},{T_{2}}\right) $.
\end{remark}



\begin{thebibliography}{99}
\bibitem{Ab} M.B. Abrahamse, \textit{Subnormal Toeplitz operators and
functions of bounded type}, Duke Math. J. \textbf{43}(1976), 597--604.

\bibitem{Ab2} M.B. Abrahamse, \textit{Commuting subnormal operators}, Ill.
J. of Math. \textbf{22(1)}(1978), 171--176.

\bibitem{AbD} M.B. Abrahamse and R.G. Douglas, \textit{A class of subnormal
operators related to multiply connected domains}, Adv. Math. \textbf{19}%
(1976), 106--148.

\bibitem{Agl} J. Agler, \textit{Hypercontractions and subnormality}, J.
Operator Theory \textbf{13(2)}(1985), 203--217.

\bibitem{Ath} A. Athavale, \textit{On joint hyponormality of operators},
Proc. Amer. Math. Soc. \textbf{103(2)}(1988), 417--423.

\bibitem{Br} J. Bram, \textit{Subnormal operators}, Duke Math. J. \textbf{22}%
(1955), 75--94.

\bibitem{Con} J.B. Conway, \textit{The Theory of Subnormal Operators},
Mathematical Surveys and Monographs, vol. 36, Amer. Math. Soc., Providence,
1991.

\bibitem{bridge} R.E. Curto, Joint hyponormality: A bridge between
hyponormality and subnormality, \textit{Proc. Symposia Pure Math.} 51(1990),
69--91.

\bibitem{QHWS} R.E. Curto, \textit{Quadratically hyponormal weighted shifts}%
, Integral Equations Operator Theory \textbf{13(1)}(1990), 49--66.

\bibitem{CLY1} R.E. Curto, S.H. Lee and J. Yoon, \textit{$k$-hyponormality
of multivariable weighted shifts}, J. Funct. Anal. \textbf{229(2)}(2005),
462--480.

\bibitem{CLY7} R.E. Curto, S.H. Lee and J. Yoon, \textit{Subnormality of
2-variable weighted shifts with diagonal core}, C. R. Acad. Sci. Paris
\textbf{351}(2013), 203--207.

\bibitem{CMX} R.E. Curto, P. Muhly and J. Xia, \textit{Hyponormal pairs of
commuting operators}, Contributions to operator theory and its applications
(Mesa, AZ, 1987), 1--22, Oper. Theory Adv. Appl. 35, Birkh\" auser, Basel,
1988.

\bibitem{CuYo1} R.E. Curto and J. Yoon, \textit{Jointly hyponormal pairs of
commuting subnormal operators need not be jointly subnormal}, Tran. Amer.
Math. Soc. \textbf{358(11)}(2006), 5139--5159.

\bibitem{De} J.A. Deddens, \textit{Intertwining analytic Toeplitz operators}%
, Mich. Math. J. \textbf{18}(1971), 243--246.

\bibitem{Fra} E. Franks, \textit{Polynomially subnormal operator tuples}, J.
Operator Theory \textbf{31(2)}(1994), 219--228.

\bibitem{GeWa} R. Gellar and L.J. Wallen, \textit{Subnormal weighted shifts
and the Halmos-Bram criterion}, Proc. Japan Acad. \textbf{46}(1970),
375--378.

\bibitem{Gou} H.W. Gould, \textit{Combinatorial identities}, Morgantown
Printing and Binding Co, Morgantown, 1972.

\bibitem{Hal2} P.R. Halmos, \textit{Normal dilations and extensions of
operators}, Summa Bras. Math. \textbf{2}(1950), 125--134.

\bibitem{Ito} T. Ito, \textit{On the commutative family of subnormal
operators}, J. Fac. Sci. Hokkaido Univ. Ser. I \textbf{14(3)}(1958), 1--15.

\bibitem{JeLu} N.P. Jewell and A.R. Lubin, \textit{Commuting weighted shifts
and analytic function theory in several variables}, J. Operator Theory
\textbf{1(2)}(1979), 207--223.

\bibitem{Lam} A. Lambert, \textit{Subnormality and weighted shifts}, J.
London Math. Soc. \textbf{14(3)}(1976), 476--480.

\bibitem{Lu3} A.R. Lubin, \textit{Weighted shifts and products of subnormal
operators}, Indiana Univ. Math. J. \textbf{26(5)}(1977), 839--845.

\bibitem{Lu1} A.R. Lubin, \textit{Extensions of commuting subnormal operators%
}, (Proc. Conf. Calif. State Univ., Long Beach, Calif., 1977), pp.115-120,
Lecture Notes in Math. 693, Springer, Berlin, 1978.

\bibitem{Lu2} A.R. Lubin, \textit{Weighted shifts and commuting normal
extension}, J. Austral Math. Soc. Ser. A \textbf{27(1)}(1979), 17--26.

\bibitem{Mla} W. Mlak, \textit{Commutants of subnormal operators}, Bull.
Acad. Pol. Sci. \textbf{19}(1971), 837--842.

\bibitem{Olin} R.F. Olin and J.E. Thomson, \textit{Lifting the commutant of
a subnormal operator}, Canad. J. Math. \textbf{31(1)}(1979), 148--156.

\bibitem{Slo} M. Slocinski, \textit{Normal extensions of commuting subnormal
operators}, Studia Math. \textbf{54}(1975), 259--266.

\bibitem{Sta} J. Stampfli, \textit{Which weighted shifts are subnormal?},
Pacific J. Math. \textbf{17}(1966), 367--379.

\bibitem{Szy} W. Szymanski, \textit{Dilations and subnormality}, Proc. Amer.
Math. Soc. \textbf{101(2)}(1987), 251--259.

\bibitem{Yos} T. Yoshino, \textit{Subnormal operator with a cyclic vector},
Tohoku Math. J. \textbf{21}(1969), 47--55.
\end{thebibliography}
\end{document}